\documentclass{amsart}
\usepackage{amsmath,amssymb, amscd, xcolor}

\newtheorem{prop}{{\bf Proposition}}[section]
\newtheorem{coro}[prop]{{\bf Corollary}}
\newtheorem{lemma}[prop]{{\bf Lemma}} 
\newtheorem{theor}[prop]{{\bf Theorem}} 
\newtheorem{ex}[prop]{{\bf Example}}

\newcommand{\ad}{\mathrm{ad\,}}
\newcommand{\tr}{\mathrm{tr}}
\newcommand{\Der}{\mathrm{Der}}

\DeclareMathOperator{\Rad}{Rad}

\begin{document}
\title[On Derived Lie algebras]{On Derived Lie algebras}
\author{Salvatore Siciliano}
\address{Dipartimento di Matematica e Fisica ``Ennio De Giorgi", Universit\`{a} del Salento,
Via Provinciale Lecce--Arnesano, 73100--Lecce, Italy}
\email{salvatore.siciliano@unisalento.it}
\author{David A. TOWERS}
\address{Lancaster University\\
School of Mathematical Sciences \\
LA$1$ $4$YF Lancaster\\
ENGLAND}
\email{d.towers@lancaster.ac.uk}


\maketitle

\begin{abstract} 
In this paper we investigate the problem of which Lie algebras appear as the derived algebra of a Lie algebra. We present new results that further develop this study and address two questions raised in a paper concerned with the corresponding problem for groups.
\end{abstract}

\section{Introduction}
Questions concerning which groups can or cannot appear as the derived subgroup of a group have been studied by many authors since Bernhard Neumann first raised the question in 1956. Groups which can so appear have been called {\it integrable}, {\it $C$-groups} or {\it commutator-realizable groups}, and a survey with many new results is given by  Ara\'{u}jo, Cameron, Casolo and Matucci in \cite{accm}. Similar questions concerning the derived algebra of a Lie algebra have also been raised (see \cite{bemn}, \cite{chao}, \cite{green}, \cite{stit} and the references therein, for example). The purpose of this paper is to carry this study further and to answer two of the questions raised in \cite{accm}.
\par

Throughout the paper, all Lie algebras $L$ are assumed to be finite-dimensional over a field $F$. No assumptions will be made on $F$ other than those specified. The {\it derived algebra} of $L$, $L^2=[L,L]$, is spanned by all products of the form $[x,y]$ with $x,y\in L$; $L$ is called {\it perfect} if $L^2=L$. If $S$ is a subset of $L$, the \emph{centraliser} of $S$ in $L$ is $C_L(S)=\{ x\in L \mid [x,H]=0\}$; the \emph{ centre} of $L$ is $Z(L)=C_L(L)$. We denote by $\Rad(L)$  the \emph{radical} of $L$, that is, the maximal solvable ideal of $L$. The terms of the descending central series are defined by $L^1=L$  and $L^i=[L^{i-1},L]$ for every $i>1$.  We denote by $L^{(i)}$ the $i$th term of the derived series of a Lie algebra $L$; that is, $L^{(0)}=L$ and, for $i \geq 1$, $L^{(i)}=[L^{(i-1)}, L^{(i-1)}]$.
\par

A \emph{composition series} for $L$ is a sequence $$0=L_0 \subset L_1\subset \ldots  \subset L_n=L,$$ where $L_i$ is a maximal ideal of $L_{i+1}$ for $0\leq i\leq n-1$; the factors $L_{i+1}/L_i$ are called \emph{composition factors}. We shall call $L$ \emph{ supersolvable} if there is a chain $$0=L_0 \subset L_1 \subset \ldots \subset L_{n-1} \subset L_n=L,$$ where $L_i$ is an $i$-dimensional ideal of $L$. An ideal $I$ of $L$ is called \emph{characteristic} if it is invariant under all derivations of $L$. We shall denote algebra direct sums by $\oplus$, whereas direct sums of the vector space structure alone will be written as $\dot{+}$. Other notation and terminology will be taken from \cite{jac}. 
\par

In Section 2 we provide some necessary conditions for a Lie algebra to be a derived algebra and use them to show that a complete Lie algebra is a derived algebra if and only if it is perfect, and to produce a necessary and sufficient condition for an almost abelian Lie algebra to be a derived algebra. We also show that a non-abelian Lie algebra whose centre is one-dimensional cannot be the derived algebra of a nilpotent algebra, but that it can be the derived algebra of a solvable algebra. We give an account of some known results concerning characteristically nilpotent Lie algebras. In addition, we collect together some straightforward preliminary results that are used later.
\par

Lie algebras with trivial centre are considered in Section 3.  Lie algebras that are the derived algebra of another algebra $H$ with $C_H(L)=0$ are also studied. In Section 4 we characterise those Lie algebras with a composition series of length 2, 3 or 4 and which are derived algebras. 
\par

Section 5 addresses two problems that have been raised in \cite{accm}. The first of these is Problem 10.20 which asks whether it is true that a Lie algebra is integrable if and only if the corresponding Lie group is. We show that this is true if the group is simply connected and has a simply connected integral. The second asks about the corresponding problem for associative rings, raised as Problem 10.21 in \cite{accm}. It is shown that the situation for rings is fundamentally different from that for groups or Lie algebras.

\section{Derived Lie algebras}

Our first result provides some necessary conditions for a Lie algebra to be a derived algebra. If $I$ is an ideal of a Lie algebra $L$ and $x\in L$, we shall write $\ad_I x$ to denote the restriction to $I$ of the adjoint map $\ad x : L\rightarrow L, h\mapsto [x,h]$. 

\begin{theor}\label{necessary}
Let $L$ be a Lie algebra over a field $F$. If $L$ is a derived algebra of some Lie algebra $H$, then the following properties hold:
\begin{enumerate} 
\item[(i)] $H$ can be chosen to be finite-dimensional;
\item[(ii)] every one-dimensional characteristic ideal is contained in $Z(L)$;
\item[(iii)] $\ad (L) \subseteq  \Der(L)^2$.
\end{enumerate} 
\end{theor}
\begin{proof} (i) Suppose that $L$ is the derived algebra of the Lie algebra $\tilde{H}$. Then $\tilde H^2=L$ is finite-dimensional, so there exist $x_1,y_1,\ldots,x_n,y_n\in H$ such that $[x_1,y_1],\ldots,[x_n,y_n]$ form a basis of $L$. Let $V$ be the subspace of $\tilde H$ spanned by  $x_1,y_1,\ldots,x_n,y_n$ and put $H=V+\tilde H^2$. Then $H$ is a finite-dimensional subalgebra of $\tilde H$ with $H^2=L$, as desired. 

(ii) Suppose, by contradiction, that there exists a one-dimensional characteristic ideal $I$ with $I\not\subseteq Z(L)$. Then $I$  is an ideal of $H$ and the set consisting of all $\ad_I h$, with $h\in H$, is a subalgebra of  $\Der(I)$.   Consequently, as $\Der(I)$ is abelian, for all $h_1,h_2\in H$ and  $x\in L$ one has
$$
[[h_1,h_2],x]=[\ad_Ih_1,\ad_Ih_2](x)=0.
$$
It follows that $[L,I]=[[H,H],I]=0$ and so $I\subseteq Z(L)$, a contradiction.

(iii) Consider the Lie algebra homomorphism $\varphi: H \longrightarrow \Der(L)$, $h \mapsto \ad_Lh$. Then $\ker \varphi=C_H(L)$ and we can identify $L/C_H(L)$ with its isomorphic image in $\Der(L)$. Moreover, we have $C_H(L)\cap L=Z(L)$ and so
$$
\left(\frac{H}{C_H(L)}\right)^2=\frac{H^2+C_H(L)}{C_H(L)}=\frac{L+C_H(L)}{C_H(L)}\cong \frac{L}{L\cap C_H(L)}=\frac{L}{Z(L)}\cong \ad(L).
$$ 
Therefore $\ad(L)$ is the derived algebra of $H/C_H(L)$, which is contained in $\Der(L)$. Moreover, we have
$$
\ad(L)=\left(H/C_H(L)\right)^2\subseteq \Der(L)^2,
$$
yielding the claim.
\end{proof}

We recall that a Lie algebra $L$ is said to be {\it complete} if $Z(L)=0$ and $\Der(L)=\ad(L)$. By the previous result we have that

\begin{coro} Let $L$ be a complete Lie algebra over any field. Then $L$ is a derived algebra if and only if it is perfect.
\end{coro}

\medskip

A Lie algebra $L$ is called \emph{almost abelian} if $L = Fx \dot{+} A$, where $A$ is an abelian ideal of $L$ and $\ad(x)$ acts as the identity map on $A$. Another immediate consequence of Theorem \ref{necessary} is the following

\begin{coro}\label{almabel} Let $L$ be an almost abelian Lie algebra of dimension $n$ over a field $F$. Then $L$ is a derived algebra if and only if the characteristic of $F$ divides $n-1$.
\end{coro}
\begin{proof} If $L$ is a derived algebra, then by Theorem \ref{necessary} we must have $\ad (L) \subseteq  \Der(L)^2$. In particular, all inner derivations of $L$ have trace zero. As $\tr(\ad x)=n-1$, it follows that the characteristic of $F$ divides $n-1$. 

We now show that the condition is sufficient. Assume then the characteristic $p$ of $F$ divides $n-1$.   Let $x,a_1,\ldots,a_{n-1}$ a basis of $L$ such that $[x,a_i]=a_i$ and $[a_i,a_j]=0$ for all $i,j=1,\ldots,n-1$. As $L$ is centreless we have $L\cong \ad(L)$, so it is enough to show that $\ad(L)$  is  a derived algebra. Note that every linear transformation $D$ of $L$ with $D(x)=0$ and $\mathrm{Im}(D)\subseteq A$ is a derivation of $L$. Let $A$ be the vector subspace spanned by the elements $a_i$, which clearly coincides with the abelian ideal of codimension 1 of $L$. Since $p$ divides $n-1=\dim A$, a well-known result of elementary linear algebra assures that the identity map $\mathrm{id}_A$ of $A$ can be expressed as the Lie commutator of two linear transformations $f_1$ and $f_2$ of $A$. Consider the linear transformations $D_1$ and $D_2$ of $L$ vanishing on $x$ and agreeing with $f_1$ and $f_2$ on $A$, respectively. As $A$ is abelian, we have $[\ad a_i,D_1](a)=[\ad a_i,D_2](a)=0$ for every $a\in A$ and for all $i=1,\ldots,n$. Now, it is easy to see that every linear map  $f:L\rightarrow A$ vanishing on $A$ is a linear combinations of the inner derivations $\ad a_i$.  Therefore  $[\ad a_i,D_j]\in \ad(L)$ for all $i=1,\ldots,n$ and $j=1,2$.  It follows that the subspace $H$ of $\Der(L)$ spanned by $D_1,D_2,[D_1,D_2]$ and the inner derivations of $L$ is a subalgebra of $\Der(L)$. From what has been observed, it is clear that $H^2\subseteq \ad(L)$. On the other hand, we have $[D_1,D_2]=\ad x$ and $[\ad x, \ad a_i]=\ad a_i$ for all $i$, which proves the opposite inclusion, yielding the desired conclusion.
\end{proof}

\begin{lemma}\label{rad} Let $L$ be a derived Lie algebra over a field of characteristic zero and $I$ the smallest term of the derived series of $L$. Then $\Rad(L)$ and $L/I$ are both nilpotent.
\end{lemma}
\begin{proof} Let $H$ be a Lie algebra such that $L=H^2$. By Levi's Theorem we have $H=\Rad(L)\dot{+}S$, where $S$ is a semisimple subalgebra of $L$. It follows that $L=H^2=\Rad(L)^2+[\Rad(L),S]+S$ and so, by \cite[Section II.7, Theorem 13]{jac}, we deduce that the radical of $L$ is $\Rad(L)^2+[\Rad(L),S]\subseteq N$, where $N$ is the nilradical of $L$. 

In order to prove the second assertion, note that $L/I$ is solvable. Moreover, as $I$ is a characteristic ideal of $L$, we have that $I$ is an ideal of $H$ and $[H/I,H/I]=L/I$. In particular, $L/I$ is a derived Lie algebra and so it is nilpotent by the first part.
\end{proof}
\medskip

Clearly, then, every non-nilpotent solvable Lie algebra over a field of characteristic zero cannot be a derived algebra. However, that is not the case in characteristic $p>0$. We call a Lie algebra {\it completely solvable} if its derived algebra is nilpotent. In characteristic zero, all solvable Lie algebras are completely solvable (by Lie's Theorem), but that is not the case in positive characteristic: for examples see \cite{bt}. Over an algebraically closed field, it was proved in \cite{dzu} that a Lie algebra is completely solvable if and only if it is supersolvable, so that a nilpotent algebra over such a field can only be the derived algebra of a supersolvable algebra.
\par

The following follows from \cite[Theorem 1]{chao}.

\begin{theor} A non-abelian Lie algebra $L$ whose centre is one-dimensional cannot be the derived algebra of a nilpotent algebra.
\end{theor}
\medskip

However, it is possible for a non-abelian Lie algebra whose centre is one-dimensional to be the derived algebra of a solvable algebra, as the following examples show.



\begin{ex}\label{ex1}
\emph{Let $F_n$ denote the  \emph{nth standard filiform Lie algebra} over a field of characteristic zero with basis $ e_1, \ldots , e_n$ and non-zero brackets $ [e_1, e_i] = e_{i+1}$ for all $i=2,\ldots ,n-1$. Then $F_n$ is the derived algebra of the Lie algebra $L_n$ with basis  $ e,e_1, \ldots , e_n$ and 
non-zero products  $ [e, e_i] = ie_i$ for all $i=1,\ldots ,n$, and $ [e_1, e_i] = e_{i+1}$ for all $i=2,\ldots ,n-1$.}
\end{ex}

\begin{ex}\label{ex2}
\emph{Let $H_n$ denote the \emph{nth Heisenberg Lie algebra} over a field of characteristic different from two with basis $ a_1,b_1 \ldots , a_n,b_n,z$ and non-zero brackets $ [a_i, b_i] = z$ for all $i=1,\ldots ,n$. Then $H_n$ is the derived algebra of the Lie algebra $L_n$ with basis  $ e,a_1,b_1 \ldots , a_n,b_n,z$ and 
non-zero products  $ [e, a_i] = a_i$, $[e,b_i]=-b_i$, $ [a_i, b_i] = z$, for all $i=1,\ldots ,n$.}
\end{ex}

The following follows from \cite[Corollary]{bokut}.
\begin{theor} A non-abelian Lie algebra $L$ with $\dim L/L^2\leq 3$ cannot be the derived algebra of a nilpotent algebra.
\end{theor}

\begin{lemma}\label{dsum} Let $L=L_1\oplus L_2$, where $L_1, L_2$ are ideals of $L$. Then
\begin{itemize}
\item[(i)] if $L_1$ and $L_2$ are derived algebras, so is $L$;
\item[(ii)] if $L$ is a derived algebra and $L_1$ is centre-free, $L_1$ is a derived algebra.
\end{itemize}
\end{lemma}
\begin{proof} (i) Let $H_1, H_2$ be Lie algebras such that $L_1=H_1^2$ and $L_2=H_2^2$. Put $H=H_1\oplus H_2$. Then $H^2=H_1^2\oplus H_2^2= L_1\oplus L_2$.
\medskip

\noindent (ii) Let $D$ be a derivation of $L$ and let $y\in L_2$. Then, for every $x\in L_1$, $$0=D([x,y])=[x,D(y)]+[D(x),y],$$  
so $[x,D(y)]=0.$ It follows that $D(y)\subseteq C_L(L_1)\subseteq L_2$ since $L_1$ is centre-free. Hence $L_2$ is a characteristic ideal of $L$.
\par

Let $L=H^2$. Then $L_2$ is an ideal of $H$ and $L_2\subseteq H^2$. Thus
\[ \left( \frac{H}{L_2}\right)^2=\frac{H^2}{L_2}\cong L_1,
\] so $L_1$ is a derived algebra.
\end{proof}

\medskip

\begin{lemma}\label{nonsing} Let $L$ be a Lie algebra over any field. If there exists a derivation of $L$ inducing a non-singular linear transformation on $L/L^2$, then $L$ is a derived algebra. In particular, this is the case when $L$ admits a non-singular derivation.
\end{lemma}
\begin{proof} Recall that $L^2$ is a characteristic ideal of $L$, so every derivation $D$ induces a well-defined linear transformation on $L/L^2$. If this linear transformation is non-singular, then for $H=FD\oplus L$ we have  $[H,H]=D(H)+L^2=L$, yielding the conclusion.
\end{proof}

An application of Lemma \ref{nonsing} yields the following:

\begin{prop}\label{suff} Let $L$  be a Lie algebra over a field $F$. In each of the following cases $L$ is a derived algebra:
\begin{enumerate} 
\item[(i)] $L$ is abelian;
\item[(ii)] $L$ is nilpotent of class 2;
\item[(iii)] $F$ has characteristic zero and  $\Rad(L)$ is abelian.
\end{enumerate} 
\end{prop}
\begin{proof} (i) As $L$ is abelian, the identity map is a non-singular derivation of $L$, so the claim follows from Lemma \ref{nonsing}.

(ii) Apply Lemma \ref{nonsing} by considering a linear transformation $D$ of $L$ inducing the identity on $L/L^2$ and such that $D(z)=2z$ for every $z\in L^2$.  

(iii) By hypothesis, we have a Levi Decomposition $L=S\dot{+} A$, where $S$ is a semisimple algebra and $A$ is an abelian ideal of $L$. Consider the linear transformation of $L$ vanishing on $S$ and inducing the identity map on $A$.  Since $L^2=S$, we see that $D$ satisfies the hypothesis of Lemma \ref{nonsing}, yielding the assertion.
\end{proof}

In 1955 Jacobson proved in (see \cite[Theorem 3]{jacobson}) that, in characteristic zero, a Lie algebra admitting a non-singular derivation is nilpotent. In the same paper he asked if, conversely, every nilpotent Lie algebra necessarily has a non-singular derivation. Eventually, this was answered in the negative by Dixmier and Lister in the celebrated paper \cite{DL} and their counterexample lead to the notion of characteristically nilpotent Lie algebras.  We recall the definition. For a Lie algebra $L$, we set
$$
L^{[1]} = \Der (L) (L) = \{D(x) \vert \, D\in \Der(L), \, x\in L \}
$$
 and 
$$
L^{[k]} = \Der (L) (L^{[k-1]}), \qquad k>1.
$$
The Lie algebra $L$ is said to be \emph{characteristically nilpotent} if there exists a positive integer $m$ such that  $L^{[m]}=0$. By the Engel-Jacobson Theorem, it is easy to see that this is equivalent to requiring  that all derivations of $L$ are nilpotent. Obviously, every characteristically nilpotent Lie algebra is nilpotent. Moreover,  in \cite[Theorem 1]{LT} Leger and  T\^og\^o proved  that $L$ is characteristically nilpotent if and only if $\Der(L)$ is nilpotent and $\dim L >1$.

The counterexample provided by Dixmier and Lister in \cite{DL} also answered in the negative the question of whether every nilpotent Lie algebra is the derived algebra of some  Lie algebra.  The same authors then asked whether a characteristically nilpotent Lie algebra can ever be a  derived algebra. In \cite{LT} Leger and T\^og\^o proved the following result, which shows that this is not the case under certain conditions:

\begin{prop}
Let $L$ be a characteristically nilpotent Lie algebra. If $\Der (L)$ annihilates $Z(L)$ then $L$ is not a derived algebra.
\end{prop}

\begin{theor} Let $L$ be a characteristically nilpotent Lie algebra. Let $n$ denote the nilpotence class of $L$ and $m$ be the smallest integers such that $L^{[m]}=0$. If $2(m - 1) > n + 1$, then $L$ is not a derived algebra.
\end{theor}

In \cite{L}, Luks eventually solved the problem in the positive by constructing a 18-dimensional Lie algebra whose derived subalgebra is characteristically nilpotent. 

We recall that a Lie algebra of dimension $n$ is said to be \emph{filiform} if it is nilpotent of class $n-1$. In \cite{CN} the following is proved:

\begin{prop} Let $L$ be a filiform Lie algebra. Then $L$ is a derived algebra if and only if it is not characteristically nilpotent.
\end{prop}




\section{Centreless Lie algebras}

If $L$ is a Lie algebra with $Z(L)=0$, then $L\cong \ad(L)$, $\ad(L)$ is an ideal of $\Der(L)$ and $C_{\Der(L)}(\ad(L))=0$, by \cite[Lemma 10]{schenk}. It follows from this that $Z(\Der(L))=0$, so we can form a tower of derivation algebras of $L$ by putting $D_0=L$, $D_i=\Der(D_{i-1})$ for $i\geq 1$. Then Schenkman proved that the sequence $D_0, D_1, \ldots$ ends after finitely many steps, and the final term in the sequence is complete. We shall show here that the same is true for Lie algebras $L$ that are the derived algebra of a reduced algebra.
\par

If $L$ is a Lie algebra with $Z(L)=0$, we say that $L$ is the derived algebra of a  \emph{reduced} algebra $H$ if $C_H(L)=0$.

\begin{lemma}\label{1} If $Z(L)=0$ and $L$ is the derived algebra of $H$, then
\begin{itemize}
\item[(i)] $C_H(L)=Z(H)$; and
\item[(ii)] $L$ is the derived algebra of the reduced algebra $H/C_H(L)$.
\end{itemize}
\end{lemma}
\begin{proof} (i)  Clearly $Z(H)\subseteq C_H(L)$. Let $h\in C_H(L)$, so $[h,x]=0$ for all $x\in L$. For every $k\in H$ we have $[k,x]\in L$ as $L$ is an ideal of $H$. Hence $$0=[h,[k,x]]=-[k,[x,h]]-[x,[h,k]]=-[x,[h,k]].$$ 
But $[h,k]\in H^2=L$ so $[h,k]\in C_H(L)\cap L=Z(L)=0$, which gives $h\in Z(H)$. Hence $C_H(L)\subseteq Z(H)$. 
\par

\noindent (ii) We have
\begin{align*}
\left(\frac{H}{C_H(L)}\right)^2&= \frac{H^2+C_H(L)}{C_H(L)}=\frac{L+C_H(L)}{C_H(L)}\cong \frac{L}{L\cap C_H(L)} \\
& =\frac{L}{Z(L)}=L,
\end{align*} so $H/C_H(L)$ has $L$ as its derived algebra. But $\theta : H \rightarrow \Der(L),  h\mapsto \ad_L h$, is a homomorphism with kernel $C_H(L)$, so $H/C_H(L)$ acts faithfully on $L$, and so it is reduced.
\end{proof}

Note that, if $Z(L)=0$ and $L$ as the derived algebra of the reduced algebra $H$, then $Z(H)=0$, so the process can be continued.

\begin{prop} Let $L$ be a finite-dimensional Lie algebra with $Z(L)=0$. Suppose that
\[
 L = L_0\subset  L_1\subset L_2 \subset \ldots,
\] where $L_n$ is the derived algebra of the reduced algebra $L_{n+1}$ for all $n\geq 1$. Then this sequence terminates after finitely many steps.
\end{prop}
\begin{proof} Note that $L_0$ is an ideal of $L_n$ for all $n\geq 1$, since it is the $n$th term of the derived series for $L_n$.
\par

We claim that $C_{L_n}(L_0)=0$ for all $n\geq 1$. We prove this by induction on $n$. It clearly holds for $n=1$, so assume the result holds for $n=k$ and let $x\in C_{L_{k+1}}(L_0)$. 
\par

For every $y\in L_{k+1}$ we have $[[x,y],L_0]=-[[y,L_0],x]-[[L_0,x],y]=0$ and $[x,y]\in L_k$, so $[x,y]\in C_{L_k}(L_0)$. It follows that $[x,y]=0$, by the inductive hypothesis, whence $x\in Z(L_{k+1})$. Therefore we have $x=0$, by Lemma \ref{1} and the construction.
\par

But now we have that $L_n$ is embedded in $\Der(L_0)$ for all $n\geq 0$, so $\dim (L_n)$ is bounded by $\dim (\Der(L_0))$ and the sequence terminates.
\end{proof}

\section{Lie algebras with composition lengths 2, 3 and 4}
This section is devoted to the classification of Lie algebras over fields of characteristic zero with composition length $\leq 4$ that are derived algebras. 
\begin{lemma}\label{comp} Let $L$ be a Lie algebra in which the top factor of every composition series is simple. Then $L$ is perfect and hence is its own derived algebra.
\end{lemma}
\begin{proof} Suppose that $L$ is not perfect. Then it has a composition series in which the top factor is one-dimensional and so not simple.
\end{proof}
\medskip

For the rest of this section we shall assume that the ground field has characteristic zero.

\begin{theor}\label{comp2} Let $L$ be a Lie algebra, and suppose that it has composition series of length 2. Then, either $L$ is a derived algebra or $L$ is the non-abelian Lie algebra of dimension 2. 
\end{theor}
\begin{proof} Suppose that $0 \subset B\subset L$ is a composition series for $L$ and that $L$ is not a derived algebra. Then, we can assume that $\dim (L/B)=1$, by Lemma \ref{comp}. Suppose that $B$ is simple. Clearly, $L$ cannot be semisimple and its radical $R$ must be one-dimensional. Hence $L=B\oplus Fx$ for some $x\in L$ and is a derived algebra by Lemma \ref{dsum}(i), a contradiction. It follows that $\dim L=2$. Now, if $L$ were abelian then it would be the derived algebra of an almost abelian Lie algebra, a contradiction. Thus $L$ is the 2-dimensional non-abelian Lie algebra, which is not a derived algebra by Corollary \ref{almabel}.
\end{proof}
\medskip

Note that the only three-dimensional Lie algebras that are derived algebras are abelian, the Heisenberg algebra and the perfect algebras; the others are solvable, but not nilpotent, and there are infinitely many of these, even over an algebraically closed field (see \cite[Chapter 1, Section 4]{jac}).

\begin{lemma}\label{cod1} Let $L$ have an ideal $B=\Rad(B)\oplus S$ of codimension one in $L$, where $S$ is a semisimple ideal of $B$. Then $L=\Rad(L)\oplus S$.
\end{lemma}
\begin{proof} Clearly, $L=\Rad(L)\dot{+} S$. Then $[\Rad(L),S]\subseteq \Rad(L)\cap B\subseteq \Rad(B)$, whence $$[\Rad(L),S]=[\Rad(L),S^2]\subseteq [[\Rad(L),S],S]\subseteq [\Rad(B),S]=0.$$
\end{proof}

\begin{theor}\label{comp3} Let $L$ be a Lie algebra, and suppose that it has composition series of length 3. Then $L$ is a derived algebra if and only if it is one of the following.
\begin{itemize}
\item[(i)] $L$ is perfect;
\item[(ii)] $L$ is three-dimensional abelian or the three-dimensional Heisenberg algebra $H_1$;
\item[(iii)] $L=S\oplus R$, where $S$ is semisimple with two simple components and $R$ is one-dimensional;
\item[(iv)] $L=S\oplus R$, where $S$ is simple and $R$ is the two-dimensional abelian Lie algebra.
\end{itemize}
\end{theor}
\begin{proof} Suppose that $0\subset C\subset B\subset L$ is a composition series for $L$ and that $L$ is not a perfect algebra. Then we can assume that $\dim (L/B)=1$ by Lemma \ref{comp}. We also have that $L$ is not semisimple, so let $R=\Rad(L)$. If $B$ is semisimple, then by Lemma \ref{cod1} we have $L=B\oplus Fx$ where $Fx=R$, and $L$ is of the form given in (iii) and is a derived algebra by Corollary \ref{almabel}.
\par

Suppose then that $B$ is not semisimple and let $\Gamma=\Rad(B)$. Then $R\cap B\subseteq \Gamma$ and, moreover, either $\dim (B/C)=1$ or $\dim C=1$. Suppose that $L$ is not three-dimensional, so $\dim \Gamma = 1$. In the former case we have $B=C\oplus \Gamma$, where $C$ is simple and $\Gamma=Fy$ for some $y\in B$. Now $\dim (L/C)=2$, so $C$ is a Levi factor and $L=C\dot{+} R$ where $R=Fx+Fy$. Moreover, $L=C\oplus R$, by Lemma \ref{cod1}. If $R$ is abelian, then $L$ is a derived algebra, by Lemma \ref{dsum} (i); if $R$ is nonabelian, then $L$ is not a derived algebra by Lemma \ref{rad}, since its radical is not nilpotent. 
\par

Finally, suppose $\dim C=1$. Then $C=\Gamma$ and $B=S\dot{+}\Gamma$ where $S$ is simple and $\Gamma=Fy$. As above, we must have that $\dim R=2$ and $L=S\dot{+} R$ where $R=Fx+Fy$. 
Now, we have $[S,Fy]\subseteq Fy$ and $[x,y]\in B\cap R=Fy$, so $Fy$ is an ideal of $L$. As one-dimensional modules of a simple Lie algebra are trivial, 
it follows that $[S,\Gamma]=0$. We now have a composition series $0\subset S\subset B\subset L$ and we are in the case covered by the previous paragraph.
\end{proof}

\begin{theor} Let $L$ be a Lie algebra, and suppose that it has composition series of length 4. Then $L$ is a derived algebra if and only if one of the following occurs.
\begin{itemize}
\item[(i)] $L$ is a perfect algebra;
\item[(ii)] $L$ is four-dimensional abelian, or isomorphic to $H_1\oplus Fx$ where $H_1$ is the three-dimensional Heisenberg Lie algebra, or to the standard filiform Lie algebra $F_4$;
\item[(iii)] $L=S\dot{+} R$, where $S$ is semisimple with one/two/three components and $R$ is abelian of dimension three/two/one, respectively;
\item[(iv)] $L=S\oplus H_1$, where $S$ is simple and $H_1$ is the three-dimensional Heisenberg algebra.
\end{itemize}
\end{theor}
\begin{proof}  Let $0\subset D\subset C\subset B\subset L$ be a composition series for $L$ and suppose that $L$ is not a perfect algebra.  By Lemma \ref{comp}  we can assume $\dim (L/B)=1$. We consider several cases as follows.
\medskip

\noindent {\bf Case 1:} All of the composition factors have dimension one. Then $L$ is a four-dimensional nilpotent by Lemma \ref{rad}. It follows that $L$ is one of the Lie algebras described in (ii), which are derived algebras in view of Theorem \ref{suff},  Example \ref{ex1} and Example \ref{ex2}.
\par

\noindent {\bf Case 2:} There is one simple factor.
\par

{\bf Case 2(a):} Suppose that $D$ is simple. Then repeated use of Lemma \ref{cod1} yields that $L=\Rad(L)\oplus D$, where $\dim \Rad(L)=3$. Therefore, by Lemma \ref{rad}, $L$ will be a derived algebra if and only if $R=\Rad(L)$ is abelian or the three-dimensional Heisenberg algebra, and we have (iii) or (iv). 
\par

{\bf Case 2(b):} Suppose that $C/D$ is simple. Then, as in the last paragraph of Theorem \ref{comp3}, $C=\Gamma\oplus S$, where $\Gamma$ is the radical of $C$ and $S\cong C/D$. Repeated use of Lemma \ref{dsum} then yields that $L$ is derived if and only if it is as described in (iii) or (iv).
\par

{\bf Case 2(c):} Suppose that $B/C$ is simple. Then $B$ is not semisimple, as it has an ideal of dimension two, so $B=\Gamma \dot{+} S$, where $C=\Gamma$ and $S\cong B/C$. Hence $L=\Rad(L)\dot{+} S$, where $\Rad(L)$ is three-dimensional. To be a derived algebra, by Theorem \ref{rad} it is necessary for $R=\Rad(L)$ to be abelian or the three-dimensional Heisenberg abelian. If $R$ is abelian, this is also sufficient, by Proposition  \ref{suff}. Hence suppose that $R$ is the Heisenberg algebra with basis $x,y,z$ where $[x,y]=z$. Then $Fz=R^2$ is a one-dimensional ideal of $L$ and so $[S,z]=0$. 
Without loss of generality we may assume that $\Gamma=Fy+Fz$ and $[S,R]\subseteq B\cap R\subseteq \Gamma$. Let $s\in S$. Then $[s,x]=\alpha y+\beta z$, $[s,y]=\gamma y+\delta z$. Now,$$0=[s,[x,y]]=-[x,[y,s]]-[y,[s,x]]=\gamma z,$$ so $\gamma =0$ and $[S,\Gamma]\subseteq Fz$. But now $$[S,R]=[S^2,R]\subseteq [S,[S,R]]\subseteq[S,\Gamma]=[S^2,\Gamma]\subseteq [S,[S,\Gamma]]\subseteq [S,z]=0.$$ It follows that $L$ is as described in (iv) and this is sufficient for $L$ to be a derived algebra.
\par

\noindent {\bf Case 3:} There are two simple factors. In this case, $L$ is not semisimple, so $L=S\dot{+} \Rad(L)$ where $S$ is semisimple with two components and $R=\Rad(L)$ is two-dimensional. Hence $L$ is a derived algebra if and only if $R$ is abelian, by Corollary \ref{almabel} and  Proposition \ref{suff}, and so is as described in (iii).
\par

\noindent {\bf Case 4:} Suppose that $L$ has three simple factors. Then $L=S\oplus \Rad(L)$ where $S$ is semisimple with three simple components and $R=\Rad(L)$ is one-dimensional. All such algebras are derived algebras and are as described in (iii).

\end{proof}

\section{Answers to a couple of open questions}
The purpose of this section is to address two problems posed  in \cite{accm}. In Problem 10.20 the authors asked the following: ``Is it true that a Lie algebra is integrable if and only
if the corresponding Lie group is?". In that paper the authors use the term  ``integrable" to mean that the algebra (respectively, group) is a derived algebra (respectively, derived group).
\par

In \cite{green}, the following result is proved.

\begin{theor}\label{green} Let $G$ be a real Lie group with Lie algebra ${\mathfrak g}$ and let $H,K$ be connected Lie subgroups with Lie algebras ${\mathfrak h},{\mathfrak k}$. Then their commutator subgroup $[H,K]$ is a Lie subgroup. Its Lie algebra is the smallest subalgebra of  ${\mathfrak g}$ containing $[{\mathfrak h},{\mathfrak k}]$ that is invariant under $\ad\, {\mathfrak h}$ and $\ad\, {\mathfrak k}$.
\end{theor}

It follows that, if $G$ is a connected real Lie group, then the Lie algebra of $G'$ is ${\mathfrak g}'$. In \cite{green} it is also pointed out that $[H,K]$ is arc-wise connected, which implies the following lemma. 
 However, no proof is given, so we include a proof offered by Eric Wofsey.

\begin{lemma}\label {wofsey} If $G$ is a connected real Lie group, then so is $G'$.
\end{lemma}
\begin{proof} The fact that $G'$ is a Lie subgroup of $G$ follows from a theorem of Yamabe (\cite{yam}). Now, for each $k$, there is a continuous map 
\[
 f : G^{2k} \rightarrow G : (g_1,g_1',\ldots,g_k,g_k') \mapsto [g_1,g_1']\ldots [g_k,g_k'].
\] Thus, for each $k$, the set of products of $k$ commutators is connected. These connected sets each contain the identity element, so the union of all of them is connected as well; that is, $G'$ is connected.
\end{proof}
\medskip

Now, using the fact that a simply connected Lie group is path-connected and hence connected, we have the following result. 

\begin{theor} If $G$ is a simply connected real Lie group, then $G$ has a simply connected integral if and only if its Lie algebra has an integral.
\end{theor}
\begin{proof} Let $G$ have simply connected integral $H$ and let ${\mathfrak h}$ be the Lie algebra of $H$. It follows from Theorem \ref{green} that $[{\mathfrak h},{\mathfrak h}]$ is the Lie algebra of $H'=G$, and so the Lie algebra of $G$ has integral ${\mathfrak h}$.
\par

Now suppose that $G$ has Lie algebra ${\mathfrak g}$ and that ${\mathfrak g}=[{\mathfrak h},{\mathfrak h}]$. Let $H$ be the unique simply connected Lie group with Lie algebra ${\mathfrak h}$, which exists by Lie's third theorem (\cite[Theorem 5.25]{hall}). Then the Lie algebra of $H'$ is $[{\mathfrak h},{\mathfrak h}]$, by Theorem \ref{green} again. Thus, $H'$ and $G$ have the same Lie algebra. Since both of these Lie groups are connected, using Lemma \ref{wofsey}, it follows from Lie's subgroups-subalgebras Theorem (see e.g. \cite[Theorem 5.20]{hall}) that $G=H'$ and $G$ has a simply connected integral.
\end{proof}

Now we deal with Problem 10.21 in \cite{accm}, namely, ``Is there a ring theoretic analogue of the results in this paper,
now taking $[a, b] = ab - ba$? Observe that in general, the commutators of all pairs of elements in a ring
form a subring, but they do not necessarily form an ideal. Therefore the ring
theory literature considers the commutator as the ideal generated by all pairs
$[a, b]$. Nevertheless, they form a right ideal if and only if they form a left ideal:
$(ab - ba)c = a(bc - cb) + (ac)b - b(ac)$.".  
\par

As a matter of fact, the situation for rings is quite different from that of groups or Lie algebras. Indeed, we shall shortly show that even  the basic fact that abelian groups have integrals does not admit a ring theoretical analogue.  We recall some terminology and notation. 

Let $R$ be a ring. We denote by $[R,R]$ the associative ideal of $R$ denoted by the Lie commutators $[x,y]=xy-yx$, for all $x,y \in R$.  We define by induction a chain of associative ideals of $R$ as follows:
\begin{align}\nonumber
\delta^{(0)}(R)&=R,\\ \nonumber
\delta^{(n)}(R)&=[\delta^{(n-1)}(R),\delta^{(n-1)}(R)]R.
\end{align}

Following Jennings \cite{J}, we say that $R$ is \emph{solvable}
 if $\delta^{(n)}(R)=0$ for some $n$. 
Clearly, if the ring $R$ is solvable, 
  then $R$ is also solvable as a Lie algebra,
  but the converse is in general not true. Note also that, if $\bar R$ is the Dorroh extension of $R$, then $\delta^{(n)}(\bar R)=\delta^{(n)}(R)=0$ for all $n>0$, and $R$ is solvable if and only so is $\bar R$.

\begin{prop} If a solvable ring $R$ is a derived ring, then $R$ is nilpotent. In particular, a nonzero commutative ring with a unity is not a derived ring.
	\end{prop} 
	\begin{proof} Suppose that $R$ is the derived ring of $S$. In view of the remark above, we may assume that $S$ is a unital ring. As $\delta^{(1)}(S)=R$ is solvable, it follows from \cite[Theorem 5.5]{J} that $R$ is nilpotent, yielding the desired conclusion.
	\end{proof}



\begin{thebibliography}{1}

\bibitem{accm} J. Ara\'{u}jo, P.J. Cameron, C. Casolo and F. Matucci, `Integrals of groups', {\it Israel J. Math.} {\bf 234} (2019), 149-178.

\bibitem{bemn} J.C. Benjumea, F.J. Echarte, M.C. M\'{a}rquez and J. N\'{u}\~{n}ez, `Links among characteristically nilpotent $C$-graded and derived filiform Lie algebras', {\it Rocky Mountain J. Math.} {\bf 35 (4)} (2005), 1081-1098.


\bibitem{bokut} L. A. Bokut: `On nilpotent Lie algebras', Algebra and Logic {\bf 10} (1971), 13--168.

\bibitem{bt} K. Bowman and D.A. Towers, `On almost nilpotent-by-abelian Lie algebras', {\it Lin. Alg. Appl.} {\bf 247} (1996), 159-167.


\bibitem{chao} C.-Y. Chao, `A nonimbedding theorem of nilpotent Lie algebras', {\it Pacific J. Math.} {\bf 22 (2)} (1967), 231-234.

\bibitem{CN} F.J. Castro-Jim\'enez J. Nunez Vald\'es:  `On characteristically nilpotent filiform Lie algebras of dimension $9$', Comm. Algebra {\bf 23} (1995),  3059--3071.


\bibitem{DL} J. Dixmier and W.G. Lister: `Derivations of nilpotent Lie algebras', \emph{Proc. Amer. Math. Soc.} {\bf 8} (1957), 155 -- 157.

\bibitem{dzu} A.S. Dzumadil'daev, `Irreducible representations of strongly solvable Lie algebras over a field of positive characteristic', {\it Math. USSR Sbornik} {\bf  51 (1)} (1985), 207-223.

\bibitem{green} L. Greenberg, `Commutator groups and algebras'.  {\it Journal of Research of the National Bureau of Standards — B. Mathematical Sciences} {\bf 73B (3)} (1969), 247-249: https://nvlpubs.nist.gov/nistpubs/jres/$73B$/jresv$73Bn3p247\_A1b$.pdf.

\bibitem{hall} B.C. Hall, `Lie groups, Lie algebras, and Representations: an elementary introduction'. Graduate texts in Mathematics, vol. 222 (2nd edition), Springer (2015), doi:10.1007/978-3-319-13467-3, ISBN 978-3319134666.

\bibitem{jac} N. Jacobson, `Lie Algebras', Interscience, New York-London (1962).

\bibitem{jacobson} {N. Jacobson}: `A note on automorphisms and derivations of Lie algebras', {\it Proc. Amer. Math. Soc.} {\bf 6} (1955), 28--283.

\bibitem{J} S.A. Jennings: `Central chains of ideals in an associative ring',  {\it Duke Math. J. } {\bf 9} (1942), 341-355.

\bibitem{L} G.F. Leger:  `A characteristically nilpotent Lie algebra can be a derived algebra?, Proc. Amer. Math. Soc. {\bf 56} (1976), 42--44.

\bibitem{LT} G.F. Leger and S. T\^og\^o: `Characteristically nilpotent Lie algebras', {\it Duke Math. J.} {\bf 26} (1959), 623 -- 628.

\bibitem{schenk} E. Schenkman, `A theory of subinvariant Lie algebras', {\it Amer. Math. J.} {\bf 73 (2)} (1951), 453-474.

\bibitem{stit} E. Stitzinger, `A nonimbedding theorem for algebras', {\it Proc. Amer. Math. Soc,} {\bf 50 (1)} (1975), 10-13.



\bibitem{yam} H. Yamabe, `On an arc-wise connected subgroup of a Lie group', {\it Osaka Math. J.} {\bf 2 (1)} (1950), 13-14.





\end{thebibliography}
\end{document}